\documentclass[12pt]{amsart}

\usepackage[top=4cm,bottom=4cm,inner=4cm,outer=4cm]{geometry}

\usepackage{comment}

\newcommand{\disp}{\displaystyle}
\newcommand{\nc}{\newcommand}
\nc{\G}{{\Gamma}} \nc{\BC}{{\mathbb C}} \nc{\BQ}{{\mathbb Q}}
\nc{\BR}{{\mathbb R}} \nc{\BZ}{{\mathbb Z}} \nc{\BP}{{\mathbb P}} \nc{\PC}{{\BP_1(\BC)}}
\nc{\BN}{{\mathbb N}} \nc{\BM}{{\mathbb M}}
\nc{\fH}{{\mathbb H}}

\nc{\mat}{{\binom{a\,\ b}{c\,\ d}}}
\nc{\U}{{\mathcal U}}
\nc{\PS}{{\mbox{PSL}_2(\BZ)}} \nc{\SL}{{\mbox{SL}_2(\BZ)}}
\nc{\SR}{{\mbox{SL}_2(\BR)}} \nc{\PR}{{\mbox{PSL}_2(\BR)}}
\nc{\GL}{{\mbox{GL}}} \nc{\PQ}{{\mbox{PGL}_2^+(\BQ)}}
\nc{\GR}{{\mbox{GL}_2^+(\BR)}} \nc{\PG}{{\mbox{PGL}_2(\BC)}}
\nc{\GC}{{\mbox{GL}_2(\BC)}}
\nc{\f}{{\mathcal{F}(\fH)}}
\nc{\Cc}{\widehat{\BC}}
\nc{\e}{{E_{\rho}(\G)}}
\nc{\g}{{\gamma}}
\nc{\vm}{{V_{\rho}(\G)}}
\nc{\oo}{{\mathcal O}}
\nc{\M}{{\mbox{M}}}
\nc{\om}{{\omega}}

\nc{\Om}{{\Omega}}
\nc{\TX}{{\widetilde{X}}}
\nc{\ol}{\overline}
\nc{\cl}{{\mathcal L}}
\nc{\ce}{{\mathcal E}}
\nc{\la}{{\lambda}}
\nc{\La}{{\Lambda}}
\nc{\cz}{{\mathcal Z}}

\newtheorem{numbered}{}[section]
\newtheorem{thm}[numbered]{Theorem}

\newtheorem{lem}[numbered]{Lemma}
\newtheorem{remark}[numbered]{Remark}
\newtheorem{prop}[numbered]{Proposition}
\newtheorem{cor}[numbered]{Corollary}

\numberwithin{equation}{section}

\newcommand{\thmref}[1]{Theorem~\ref{#1}}
\newcommand{\propref}[1]{Proposition~\ref{#1}}

\newcommand{\lemref}[1]{Lemma~\ref{#1}}

\begin{document}

\title[]{Automorphic Schwarzian equations and integrals of weight 2 forms}
\author[]{Abdellah Sebbar} \author[]{Hicham Saber}
\address{Department of Mathematics and Statistics, University of Ottawa, Ottawa Ontario K1N 6N5 Canada}
\address{Department of Mathematics, University of Ha'il, Kingdom of Saudi Arabia}
\email{asebbar@uottawa.ca}
\email{hicham.saber7@gmail.com}
\maketitle
\maketitle
\begin{abstract}
	In this paper, we investigate the non-modular solutions to  the Schwarz differential equation	 $\{f,\tau \}=sE_4(\tau)$ where $E_4(\tau)$ is the weight 4 Eisenstein series and $s$ is a complex parameter. In particular, we provide explicit solutions for each  $s=2\pi^2(n/6)^2$ with $n\equiv 1\mod 12$. These solutions are obtained as integrals of meromorphic weight 2 modular forms. As a consequence, we find explicit solutions to the differential equation $\ \disp y''+\frac{\pi^2n^2}{36}\,E_4\,y=0$ for each $n\equiv 1\mod 12$ generalizing the work of Hurwitz and Klein on the case $n=1$. Our investigation relies on the theory of equivariant functions on the complex upper half-plane. This paper supplements a previous work where we determine all the parameters $s$ for which the above Schwarzian equation has a modular solution.
\end{abstract}
\section{Introduction}

Let $\fH$ be the upper half of the complex plane, and define the Dedekind eta function by
 \[
\eta(\tau)\,=\,q^{1/24}\prod_{n\geq 1}\,(1-q^n)\,, \ q=\exp(2\pi i\tau),\ \tau\in\fH
\]
and  the weight 4 Eisenstein series by 
\[
E_4(\tau)\,=\,1+240\,\sum_{n=1}^{\infty}\,\sigma_3(n)\,q^n\,, \ \mbox{with } \disp \sigma_3(n)=\sum_{0<d\mid n}\,d^3.
\]

In 1889, Hurwitz \cite{hurwitz} wrote down a solution to the differential 
equation
\[
y''+\frac{\pi^2}{36}\,E_4\,y=0
\]
as $y=\eta^{-2}$. 
  This equation was considered by Klein ten years earlier in \cite{klein} and much later by Van der Pol in \cite{vdp}. One of the goals of this paper is to give explicit solutions to the equation
\[
y''+\frac{\pi^2n^2}{36}\,E_4\,y=0
\]
for infinitely many integer values of $n$, namely for all $n\equiv 1\mod 12$.
The main tool in our study is to try to solve the Schwarzian differential equation attached to this differential equation which we now explain.

The Schwarz derivative or the Schwarzian of a meromorphic function $f$ on a domain $\Omega$ is given by
\begin{equation}\label{schwarz}
\{f,z\}\,=\, \left(\frac{f''}{f'}\right)'-\frac12 \left(\frac{f''}{f'}\right)^2\,=\,
\frac{f'''}{f'}-\frac32\left(\frac{f''}{f'}\right)^2.
\end{equation}
It defines a projectively invariant differential operator that behaves as a quadratic differential
\[
	\{f,z\}dz^2\,=\,\{f,w\}dw^2+\{w,z\}dz^2. 
\]
Moreover, $\{f,z\}=0$ if and only if $f$ is a linear fraction and if $\{f,z\}=\{g,z\}$ for two meromorphic functions $f$ and $g$ on $\Omega$, then $f$ is a linear fraction of $g$.

 When $\Omega$ is the complex upper half-plane $\fH$, it is natural to ask when does $\{f,\tau\}$, $\tau\in\fH$, enjoy any automorphic properties. It turns out that if $f$ is an automorphic function (of weight 0) on a discrete subgroup $\G$ of $\SR$, then $\{f,\tau \}$ is a (meromorphic) automorphic form of weight 4 for $\G$ thanks to the quadratic differential property. 
 Conversely, if $F(\tau)=\{f,\tau \}$ is a weight 4 automorphic form, then $f$ is not necessarily an automorphic function, nevertheless, there exists a complex 2-dimensional representation $\rho$ of $\G$ such
 \[
 f(\gamma\cdot \tau)\,=\,\rho(\gamma)\cdot f(\tau)\,,\ \ \tau\in\fH\,,\ \gamma\in \G,
 \]
 where the action is by linear fractions on both sides of this relation.
We say that $f$ is $\rho-$equivariant. If $\rho=1$ is constant, then $f$ is an automorphic function, and if $\rho=Id$ is the standard representation, then $f$ commutes with the action of $\G$ and we say that $f$ is an equivariant function on $\fH$.  This class of functions has been studied extensively with many applications in \cite{brady, sb1, sb2, ell-zeta,ss3, ss4, ss1}.

In the paper \cite{s-s-part1} we focused on the case $\G=\SL$ and  we studied the meromorphic functions $h$ for which the Schwarz derivative is a holomorphic modular form of weight 4. In other words, we considered the Schwarzian equation
\begin{equation}\label{the-equ}
\{h,\tau\}\,=\,s\,E_4(\tau)\,,\ \ \tau\in\fH,
\end{equation}
where $s$ is a complex parameter. If a solution $h$ is given, any other solution is a linear fraction of $h$. The main result in the paper loc. cit. determines precisely the values of the parameter $s$ for which the solution to \eqref{the-equ} is a modular function for a finite index subgroup of $\SL$. In fact, there are infinitely many such values given by $s=2\pi^2(m/n)^2$ with $\gcd(m,n)=1$ and $2\leq m\leq 5$. Also, it turns out that the representation attached to such solutions are all irreducible.

In this paper, we are interested in other type of solutions to \eqref{the-equ} for which the representations are reducible. The solutions are not modular but they arise as integrals of weight 2 modular forms with a character of $\SL$. These modular forms must be non-vanishing and if they have poles in $\fH$, they must be double poles with zero residues. In the meantime, this occurs precisely when $s=2\pi^2(m/6)^2$ with $m\equiv 1\mod 12$. The case of a holomorphic weight 2 modular form is unique and is given by $\eta^4$ where $\eta$ is a the Dedekind eta-function and it corresponds to $m=1$. Starting with this solutions, we build other solutions by considering the meromorphic modular forms
\[
f_n=\frac{\eta^4}{\prod_{i=1}^n(J-a_i)^2}\,,
\]
where $n$ is a non-negative integer and $J$ is the classical elliptic  modular invariant. If $h_n$ is the integral of $f_n$, then the Schwarz derivative of $h_n$ is holomorphic on $\fH$ if it is meromorphic with at  most simple poles. Thus, the residues of $f_n$ must vanish which is equivalent to the fact that the parameters $a_i$, $1\leq i\leq n$, satisfy a system of algebraic equations. We prove that this system has always a solution and thus, for each $n\geq 0$, we construct an explicit solution to \eqref{the-equ} with $ \disp s=2\pi^2\,\frac{(12n+1)^2}{36}$. In terms of ordinary differential equations, our construction allows us to find explicit solutions to 
\[
y''\,+\,\frac{\pi^2(12n+1)^2}{36}\,E_4\,y\,=\,0\,, \ \ n\in\BN,
\]
of the form
\[
y\,=\, \eta^{-2}\prod_{i=1}^n (J-a_i) \,.
\]
This generalizes the solution given by   Hurwitz   in the case $n=0$.

\tableofcontents

\section{Schwarzian equations and automorphy}
The Schwarz derivative is closely related to second degree differential equations. Indeed, let $\Omega$ be a domain in $\BC$ and let $F(z)$ be a meromorphic function on $\Omega$. If $ y_1$ and $y_2$ are two linearly independent solutions to the ODE
\[
y''+\frac{1}{2}F(z)\,y\,=\,0\,,
\]
then $f=y_2/y_1$ is a solution to the Schwarz differential equation
\[
\{f,z\}\,=\,F(z).
\]
Most of the functional properties of the Schwarz derivative follow from  this connection.

 Now we suppose that $\Omega=\fH$, the complex upper half-plane, in which case we will  use the variable $\tau$ as a variable instead of $z$. Let $F$ is a nonzero weight 4 automorphic form for a discrete subgroup $\G$ of $\SR$ and suppose there exists
   a meromorphic function $f$ on $\fH$ such that $\{f,\tau\}=F(\tau)$. For $\gamma\in\G$ and $\tau\in\fH$, set $g(\tau)=f(\gamma \tau)$ (Here, whenever  $\disp \gamma=\mat$ is any matrix, and $z$ is a complex number, then $\disp \gamma z:=\frac{az+b}{cz+d}$). We have
   \[
   \{g,\tau\}=\{g,\gamma \tau\}\,\frac{d^2\gamma \tau}{d\tau^2}+\{\gamma,\tau\}.
   \]
   As $~\disp \frac{d^2\gamma 
   }{d\tau^2}=\frac{1}{(c\tau+d)^4}$ and $\{\gamma \tau,\tau\}=0$, we have
   \[
   \{g,\tau\}=\frac{1}{(c\tau+d)^4}\,F(\gamma \tau)=F(\tau).
   \]
   Therefore $\{g,\tau\}=\{f,\tau\}$ and consequently $g$ is a linear fraction of $f$. This defines a representation $\rho$ of $\G$ in $\PG$ such that
   \[
   f(\gamma \tau)=\rho(\gamma) f(\tau)\,,\ \ \tau\in\fH\,\ \ \gamma\in\G,
   \]
and we say that $f$ is $\rho-$equivariant. Conversely, if we are given a pair $(\Gamma,\rho)$ where $\Gamma$ is a discrete subgroup of $\SR$ and $\rho$ is a representation of $\G $ in $\PG$, then $\{f,\tau\}$ is a weight 4 automorphic form for $\G $. We note that for any Fuchsian group $\Gamma$ (first or second kind) and any 2-dimensional representation $\rho$ of $\Gamma$, $\rho-$equivariant functions always exist \cite{ss4}.

As for the analytic behavior, from the definition, $\{f,\tau\}$ will have a double pole at points of $\fH$ where $f'$ vanishes or has multiple poles, and will be holomorphic elsewhere including at simple poles.

We now focus on the case $\G =\SL$ and we look for the meromorphic functions $f$ on $\fH$ such that $\{f,\tau\}$ is a weight 4 modular form that is  holomorphic on $\fH$ and at $\infty$. In other words, we would like to solve the Schwarzian equation
\begin{equation}
\label{main}
\{f,z\}\,=\,s\,E_4(z),
\end{equation}
where $s$ is a complex parameter, and $E_4$ is the weight 4 Eisenstein series.
A solution $f$ to this equations would have a non-vanishing derivative, that is, $f$ is locally univalent where it is finite and has at most simple poles. Moreover, such solution always exists. Indeed, as $H$ is simply connected and $E_4$ is holomorphic on $\fH$, the Fuchsian differential equation
\begin{equation}\label{main2}
y''\,+\,\frac{s}{2}\,E_4\,y\,=\,0
\end{equation}
has two linearly independent holomorphic solutions on $\fH$, $y_1$ and $y_2$ and $f$ is given by their ratio. Conversely, if we are given a solution to \eqref{main}, then $f'$ is non-vanishing and has only double poles (if any), therefore $\sqrt{f'}$ defines a meromorphic functions on $\fH$ with at most simple poles. One can easily verify that $y_1=1/\sqrt{f'}$ and $y_2=f/\sqrt{f'}$ are two linearly independent solutions to \eqref{main2}.

\section{Modular and non-modular solutions}
In this section, we will analyze the modular properties of a solution to the main equation \eqref{main}. As $s=0$ yields linear fractions as solutions, we will always assume $s\neq 0$. Let $r$ be a complex number with a non-negative real part such that $s=2\pi^2 r^2$ and set $q=\exp(2\pi i \tau)$. In terms of $q$, the 
 differential equation \eqref{main2} becomes
 \begin{equation}\label{main3}
 	\frac{d^2y}{dq^2}+\frac{1}{q}\frac{dy}{dq}-\frac{r^4}{4}\frac{E_4(q)}{q^2}\,y\,=\,0,
 		\end{equation}
 	where $q$ is in the punctured disc $\{0<|q|<1\}$. It is clear that we are dealing with a Fuchsian differential equation with a regular singular point at $q=0$ as $E_4(q)=1+\mbox{O}(q)$. We now apply the Frobenius method to write down the shape of the $q-$expansion of the solutions. The indicial equation $x^2-r^2/4=0$ has solutions $x_1=r/2$ and $x_2=-r/2$ and  $\mbox{Re}\,(x_1)\geq \mbox{Re}\,(x_2)$ because we assumed Re$\,(r)\geq 0$. Therefore, the first solution to \eqref{main2} has the shape
 	\[
 	y_1(q)\,=\,q^{r/2}\,\sum_{n=0}^{\infty}\,c_nq^n\,,\ \ c_0\neq0.
 	\]
 To write down the shape of the second solution, one needs to look at the difference $x_1-x_2=r$. In particular, if $r$ is not an integer, then we have another solution of the form
 \[
  	y_2(q)\,=\,q^{-r/2}\,\sum_{n=0}^{\infty}\,c^*_nq^n\,,\ \ c^*_0\neq0.
 \]
 Notice that $y_1$ and $y_2$ are linearly independent so that in the case $r\notin\BZ$, a solution to \eqref{main} is given by
 \[
 h(\tau)\,=\,\frac{y_2(q)}{y_1(q)}\,=\,q^{-r}\,\sum_{n=0}^{\infty}\, a_nq^n\,,\ \ a_0\neq 0,
\]
which is meromorphic at $\infty$. However, if $r\in\BZ$ then a second solution is given by 
\[
k\log(q) y_1(q)+q^{-r}\,\sum_{n=0}^{\infty}\, C_nq^n\,,\ \ C_0\neq0\,,\ k\in\BC,
\]
which, together with $y_1(q)$, gives a solution to \eqref{main} with a logarithmic singularity at $\infty$  of the form 
\[
h(\tau)=k\log(q)+q^{-r}\,\sum_{n=0}^{\infty}\, b_nq^n\,,\ \ b_0\neq0.
\]
Here $\log$ can be any branch of the logarithm. Since the Schwarz derivative is invariant under linear fractional transformations, we conclude the following
\begin{prop}\cite{s-s-part1} The equation $\{h,\tau\}=s\,E_4(\tau)$ has a solution of the form
	\begin{equation}
	h(\tau)\,=\,q^{r}\,\sum_{n=0}^{\infty}\, a_nq^n\,,\ \ a_0\neq 0 \ \mbox{ if  }\ r\notin\BZ 
	\end{equation}
	and
		\begin{equation}
	h(\tau)\,=\,\tau + q^{-r}\,\sum_{n=0}^{\infty}\, b_nq^n\,,\ \ b_0\neq 0 \ \mbox{ if  }\ r\in\BZ .
	\end{equation}
	\end{prop}
As a consequence,  a solution $h$ satisfies $h(\tau+m)=h(\tau)$, $\tau\in\fH$ for some integer $m$ if and only if $r\in\BQ\setminus\BZ$. Notice that this condition means that $T^m\in\ker\rho$ where $\disp T=\binom{1\ \ 0}{0\ \  1}$ and $\rho$ is the 2-dimensional representation attached to the equivariant function $h$.

In the paper \cite{s-s-part1} we investigated  which parameter $s$,  the subgroup $\ker \rho$ has  a finite index inside the modular group. In other words, for which $s$, the solution to \eqref{main} is a modular function (for a finite index subgroup of $\PS$). In fact, we have
\begin{thm} \cite{s-s-part1} The equation $\{f,\tau\}=s\,E_4(\tau)$ has a modular solution $h$  if and only if $s=2\pi^2 r^2$ where$r=m/n$ with $2\leq n\leq 5$ and $\gcd(m,n)=1$, in which case
	the invariance group of $h$ is $\Gamma(n)$.
	\end{thm}
A key idea in the proof is that if $h$ is a modular function then the associated representation $\rho$ must irreducible, and  $\ker\rho$ is a normal, genus zero and torsion free subgroup of $\PS$. If $X(\Gamma)$ denotes the modular curve  attached to a finite index subgroup $\G$, then we have two covering of compact Riemann surfaces $h: X(\ker\rho)\longrightarrow \PC$ and the natural covering $\pi: X(\ker\rho)\longrightarrow X(\PS)$. It turns out that each covering has the same ramification index at its ramified points; the first one because $h$ is $\rho-$equivariant and the second one because $\ker\rho$ is normal in $\PS$. The integer $n$ (resp. $m$) in the theorem is the common ramification index of the first (resp the second) covering.

Essentially, in this paper, we will focus on the case of the solutions $h$ to \eqref{main} such that the corresponding representation is reducible. According to \cite{s-s-part1}, these solutions are necessarily not modular. Yet, they will  arise from modular objects.

\section{Integrals of weight two modular forms}
In the previous section, we have encountered modular functions for subgroups of $\SL$ that were $\rho-$equivariant for some representation $\rho$.  In the meantime, the derivative of a modular function is a   weight two modular form. However,  not every weight 2 modular form is the derivative of a modular function. In this
 section, we will attempt to construct equivariant functions by integrating  weight 2 modular forms.
\begin{prop}
	Let $f$ be a weight 2 holomorphic modular form for a finite index subgroup $\G $ of $\SL$ and let
	\[
	h(\tau)\,=\,\int_i^{\tau}\,f(z)\,dz\,.
	\]
	Then $h$ is $\rho-$equivariant for $\G$ with $\rho$ being a triangular representation.
\end{prop}
\begin{proof}
Let $f$ be a weight 2 modular form for $\G$ and 	let $\gamma\in\G $. As
the differential form $f(z)dz$ is invariant under $\G $, we have
\[
f(\gamma\cdot \tau)\,=\,\int_i^{\gamma\cdot \tau}\, f(z)\,dz\,=\,
\int_{\gamma^{-1}\cdot i}^{\tau}\, f(z)\,dz\,=\,f(\tau) +\omega_{\gamma},
\]
where
\[
\omega_\gamma\,=\,\int_{\gamma^{-1}\cdot i}^i\,f(z) dz,\ \ \gamma\in\Gamma.
\]
Therefore, $f$ is a $\rho-$equivariant function for $\Gamma$ where
\[
\rho(\gamma)\,=\, \binom{1\ \ \omega_\gamma}{0\ \ 1}.
\]
It is easy to see that 
$
\omega_{\alpha\beta}=\omega_{\alpha}+\omega_{\beta}.
$
	\end{proof}
 In particular, if we integrate a weight two modular form that is non-vanishing
 on $\fH$, then its integral is locally univalent and its Schwarz derivative is a weight 4 holomorphic form for $\G$. Moreover, if $m$ is the level of $\G$ so that $T^m\in\G $, then $T^m\in \ker\rho$ or, in other words, $\omega_{T^m}=0$. Indeed, if $q=\exp(2\pi i\tau/m)$ then $h$ has a $q-$expansion that converges uniformly on compact subsets of $\fH$. In particular to compute $\omega_{T^m}$, we can integrate term-wise between $i-m$ and $i$ which amounts to zero.

 To obtain examples of such forms, we focus on the principal congruence subgroups $\G(n)$,  $n>1$.  In fact, from
 \cite{gunning} we have 
 $\dim M_2(\Gamma(2))=2 $ and
 \[
 \dim M_2(\Gamma(n))=\frac{n+6}{24}n^2\prod_{p\mid n,\, p \,{\tiny\mbox{prime}}}\left(1-\frac{1}{p^2}\right) \mbox{ for } n\geq 3.
 \]
 To construct non-vanishing weight 2 forms, we will make use of products of the Dedekind eta function. It satisfies transformation rules under the action of the modular group and we will use a version that does  not involve the Dedekind sums \cite{cohen}:
 \begin{equation}\label{eta1}
 \eta(\tau+1)=e^{\pi i/12}\eta(\tau)\,,
 \end{equation}
 and if $\gamma=\binom{a\ \, b}{c\ \,d}\in\SL$ with $c\neq 0$, we have
 \begin{equation}\label{eta2}
 \eta\left(\frac{a\tau+b}{c\tau+d}\right)\,=\,v(\gamma)(c\tau+d)^{1/2}\eta(\tau),
 \end{equation}
 where 
 \[
 v(\gamma)=\left(\frac{d}{|c|}\right)\exp\left(\frac{\pi i}{12}((a+d-3)c-bd(c^2-1))\right) \quad \mbox{ if }c\mbox{ is odd},
 \]
 and if $c$ is even:
 \[
 v(\gamma)=\left(\frac{c}{|d|}\right)\exp\left(\frac{\pi i}{12}((a+d-3)c-bd(c^2-1)+3d-3)\right) \varepsilon(c,d) ,
 \]
 where $\left(\frac{c}{d}\right)$ is the Legendre symbol and
 $\varepsilon(c,d)=-1$ when $c< 0$ and $d<0$ and $\varepsilon(c,d)=0$ otherwise.. In other words, $\eta$ is a weight 1/2 modular form with the multiplier system $v$ extended  to $\gamma$ with $c=0$ following \eqref{eta1}. In fact, $v$ is a character of the metaplectic group but $v^4$ is actually a character of $\SL$.
 
One can construct  non-vanishing weight 2 modular forms for $\G (n)$ for infinitely many integers $n\geq1$ as follows
 \begin{prop}\label{etaprod} Let $n\geq1$ be an integer. We have:
 	\begin{enumerate}
 		\item 	If $2|n$, then $\displaystyle \frac{\eta(\tau/n)^8}{\eta(2\tau/n)^4}\in M_2(\Gamma(n))$.
 		\item If $3|n$, then $\displaystyle \frac{\eta(\tau/n)^6}{\eta(3\tau/n)^2}\in M_2(\Gamma(n))$.
 		\item If $n\geq 1$, then $\eta^2(n\tau)\eta^2(\tau/n)\in M_2(\Gamma(6n))$.
 	\end{enumerate}
 	
 \end{prop}
 \begin{proof} Straightforward verification using \eqref{eta1} and \eqref{eta2}. 
 	\end{proof} 
 One can probably give explicit formulas in the case of odd $n$ depending on the primes dividing $n$. For instance if there is a prime divisor $p$ of $n$ such that $p\equiv 11\mod 12$, then $\eta^2(\tau/n)\eta^2(p\tau/n)\in M_2(\G(n))$ etc. Some of the eta-quotients in the above proposition were inspired by \cite{r-w} where similar forms were given for $\G_0(n)$.

 The integral of each of these functions is a $\rho-$equivariant function for $\Gamma(n)$ with $\rho$ upper triangular,  and has a holomorphic Schwarz derivative with leading coefficient $2\pi^2r^2$ but not necessarily a multiple of $E_4$. As we will see below, none of these will yield a $\rho-$equivariant function for $\SL$ except for $n=1$ in the case
 (3) in the above proposition.
 
  We now suppose we are given a 2-dimensional representation $\rho$ of a finite index subgroup $\G $ of  $\SL$ and we suppose that $\rho$ is reducible. There exists a constant matrix $\sigma\in\GC$ such $\rho_1=\sigma\rho\sigma^{-1}$ is (upper)  triangular. In the meantime,  the function $g=\sigma\cdot f$ is $\rho_1-$equivariant and has the same Schwarz derivative as $f$. In addition, the kernels of both representations are the same. Therefore, in the context of solving Schwarzian differential equations, and without loss of generality, we will assume that $\rho$ is (upper) triangular.

\begin{thm}\label{triang}
	Let $h$ be a meromorphic function on $\fH$ and let $\G $ be a finite index subgroup of $\SL$. Then $h$ is $\rho-$ equivariant for a triangular representation of $\G$ if and only if the derivative $h'$ is a meromorphic  weight $2$ modular form with a character of $\G$.
\end{thm}
\begin{proof}
	Suppose that $h$ is a $\rho-$equivariant function for $\G$ with $\rho$ triangular. 	If $\displaystyle \gamma=\mat\in\G $ write $\displaystyle 
	\rho(\gamma)=\binom{a(\gamma)\ \ b(\gamma)}{0\ \ \ \  d(\gamma)}$. Notice that $a(\gamma)$ and $b(\gamma)$ are both characters of $\G$ and so is $\chi(\gamma)=a(\gamma)/d(\gamma)$. Now differentiating the equivariance relation 
	\[\disp h(\gamma z)=\rho(\gamma)h(z)=\frac{a(\gamma)}{d(\gamma)}h(z)+\frac{b(\gamma)}{d(\gamma)}=\chi(\gamma)h(z)+\frac{b(\gamma)}{d(\gamma)}\]
	 yields
	\[
	\frac{1}{(cz+d)^2}\,h'(\gamma z) \,=\, \chi(\gamma)\,h'(z),
	\]
	that is, $h'\in M_2(\G,\chi)$ with the behavior of $h'$ contingent on that of $h$.
	
	Conversely, suppose that  $h'$ is  a weight 2 modular form for $\G$ with a character $\chi$. Write
	\[
	h(\tau)=\int_i^{\tau}\,h'(z)\,dz +h(i),
	\]
	so that for $\gamma\in\SL$, we have
	\begin{align*}
	h(\gamma \tau)&=\int_i^{\gamma \tau}\,h'(z)\,dz+h(i)\\
	&=\int_{\gamma^{-1}i}^{\tau}\,h'(\gamma z)\,d\gamma z +h(i)\\
	&=\chi(\gamma)\,\int_{\gamma^{-1}i}^{\tau}\,h'(z)\,dz+h(i)\\
	&=\chi(\gamma)h(\tau)+\chi(\gamma)\omega_{\gamma}+(1-\chi(\gamma))h(i),
	\end{align*}
where $\displaystyle \omega_{\gamma}=\int_{\gamma^{-1}i}^i\, h'(z)\,dz $.
If, for $\gamma\in\G$, we define
\begin{equation}\label{expr-rho}
\rho(\gamma)\,=\, \begin{pmatrix} 1 & \omega_{\gamma}+(\chi(\gamma^{-1})-1)h(i)\\
0 & \chi(\gamma^{-1}) \end{pmatrix}
\end{equation}
then $h$ is $\rho-$equivariant for $\G$.
	\end{proof}
This essentially characterizes the $\rho-$equivariant functions for reducible representations. For a general weight 2 meromorphic modular form $g$, its integral is well defined everywhere  if the residues at the poles are zero. Additionally,
in order for  the integral  to have a holomorphic Schwarz derivative on $\fH$, one needs to have $g$ non-vanishing and with at most poles of order 2 so that the integral will have at most simple poles. 

\begin{remark}{\em
The weight 2 modular form with a character has always a $q-$expansion in $q=\exp(2\pi i\tau/6)$ as we will see in the next section. If this expansion does not have a constant term, then its integral will  be meromorphic at $\infty$, while if it has a constant term, then its integral will have a logarithmic singularity at $\infty$ as was predicted by the Frobenius method.
}
\end{remark}
\section{Holomorphic solutions}
From \propref{etaprod}, we see that $\eta^4$ is a non-vanishing weight 2 holomorphic modular form for $\G(6)$. It is also a weight modular form for $\SL$ with a character $\chi$.
If we define $f_6$ by
\begin{equation}\label{f6}
f_6(\tau)=\int_{i}^{\tau}\eta(z)^4 dz,
\end{equation} 
then, according to \thmref{triang}, $f_6$ is $\rho-$ equivariant function for $\SL$ with $\rho$ being triangular, and as its derivative is non-vanishing, its Schwarz derivative is a weight 4 holomorphic form for $\SL$. One can determine the leading coefficients by an easy computation to deduce the following
\begin{prop}
	We have
	\begin{equation}\label{f6s}
	\{f_6,\tau\}\,=\,\frac{2\pi^2}{36}\,E_4(\tau).
	\end{equation}
\end{prop}
 It is worth mentioning that \eqref{f6s} can be deduced more directly. 
	Recall the weight 2 Eisenstein series $E_2$  defined by
	\[
	E_2(\tau)\,=\,\frac{12}{\pi i}\frac{\eta'(\tau)}{\eta(\tau)}=1-24\sum_{n\geq 1}\,\sigma_1(n)q^n\,,\ \ q=\exp(2\pi i\tau),
	\]
	so that
	\[
	E_2(\tau)=\frac{3}{\pi i}\frac{f_6'(\tau)}{f_6''(\tau)}.
	\]
	In the meantime, $E_2$ 
	satisfies  the Ramanujan identity \cite{ram}
	\begin{equation}\label{ram1}
	\frac{1}{2\pi i}\frac{dE_2(\tau)}{d\tau}=E_2(\tau)^2-E_4(\tau).
	\end{equation}
Using the definition of the Schwarz derivative, the identity \eqref{f6s} follows.

Thus, the equivariant function $f_6$ is a solution to \eqref{main} with $\displaystyle s=2\pi^2/36$. The denominator 36 is not there by chance. It is rather closely associated to the reducible representation or to the weight 2 modular forms with a character of $\SL$. Indeed, such a character must be trivial on the commutator group $\G'$ of $\SL$. This commutator is a level 6 normal congruence subgroup of index 12, and $\chi$ is completely determined by the image a generator of the coset group. In particular $\chi^{12}=1$. Thus, if $f$ is a weight 2 modular form with a character of $\SL$, then $f$ has always a $q-$ expansion with $q=\exp(2\pi i /6)$. 

\begin{prop}
	Let $g$ be a non-vanishing  weight 2 holomorphic modular form with a  character $\chi$ for $\SL$, then $g=c\eta^4$ for some constant $c$.
\end{prop}
\begin{proof}
	Since $\chi$ has order dividing 12, we see that $g^{12}$ is a non-vanishing holomorphic modular form of weight 24 and from the valence formula for modular forms, $g$ has a double zero at $\infty$. Therefore $g^{12}/\Delta^2=g^{12}/\eta^{48}$ is a  modular function
	that is holomorphic on $\fH$ and at $\infty$. It follows by Liouville theorem that $g^{12}/\eta^{48}$ is a constant. We deduce that $g=c\eta^4$ for some constant $c$.
	\end{proof}
\begin{cor}
	The only solution to \eqref{the-equ} arising as the integral of a holomorphic weight $2$ modular form for $\SL$ is $f_6$ (up to a constant factor).
\end{cor}

It follows that  none of the weight 2 modular forms listed  in \propref{etaprod} beside $\eta^4$ is a modular form with a multiplier system for $\SL$. While they behave well with the matrix $\displaystyle\binom{0\ \ -1}{1\ \  \ \ 0}$, they are in no way modular with respect to $\displaystyle\binom{1\ \ 1}{0\ \ 1}$.

We will see in the coming sections how to construct an infinite family of solutions to \eqref{main} arising as integrals of meromorphic weight 2 modular forms with a character.

The solution $f_6$ has some interesting historical perspective. From the above we see that $1/\sqrt{f_6'}=\eta^{-2}$ is a solution to the differential equation
\[
y''\,+\,\frac{\pi^2}{36}E_4 y\,=\,0,
\]
This equation has been of interest to Hurwitz in \cite[Equation 13]{hurwitz}, to Klein \cite{klein} and to Van der Pol \cite{vdp}. The following sections will generalize this solution by providing explicit solutions to differential equations of the form
\[
y''\,+\,\frac{\pi^2(12k+1)^2}{36}\,E_4\,y\,=\,0\; ,\ \ k\in\mathbb N,
\]
as well as solving the corresponding Schwarzian equations.

\section{A system of algebraic equations}
In this section we study a system of algebraic equations for which the existence of a solution will allow us to construct solutions to \eqref{main} for an infinite family of the parameters $s$. 

Let $a$, $b$ and $c$ be three positive real numbers and fix an integer $n\geq 1$. Consider the system in $n$ variables $x_1,\ldots,x_n$:
\begin{equation}\label{syst}
\frac{a}{1-x_i}\,-\,\frac{b}{x_i}\,=\,\sum_{j\neq i}\,\frac{c}{x_i-x_j}\,,\ \ 1\leq i\leq n,
\end{equation}
with the understanding that the right hand-side is zero when $n=1$. The set of solutions, if nonempty,  is acted upon by the symmetric group $S_n$. In the following, we will prove the existence of a solution $(x_1,x_2,\ldots,x_n)$.

For each $i=1,\ldots,n$, define the function
\[
f_i(x_1,\ldots,x_n)\,=\,\frac{a}{1-x_i}\,-\,\frac{b}{x_i}\,-\,\sum_{j\neq i}\,\frac{c}{x_i-x_j}.
\]
Let $H$ be the union of the hyperplanes $x_i=0$, $x_i=1$, $x_i=x_j$, $1\leq i,\,j\leq n$,  and define the function
\[
f:{\BR}^n\setminus H\longrightarrow {\BR}^n
\]
whose components are the $f_i$'s. Now let
\[
U=\{(x_1,\ldots,x_n)\in\BR^n:\, 0<x_1<x_2<\ldots<x_n<1\}.
\]
It is a connected component of the domain of $f$, and most importantly it is open and bounded so that its boundary $\partial U$ is compact.
\begin{lem}\label{lem61}
	Let $v$ a point of the boundary $\partial U$. For each sequence $(u_n)$ of $U$ converging to $v$, there exists $i\in\{1,\dots, n\}$ such that  $\disp \lim_{n\rightarrow\infty} |f_i(u_n)|=\infty$.
\end{lem}
\begin{proof}
	The boundary $\partial U$  consists of the hyperplanes $x_j=0$, $x_j=1$ and  $x_j=x_{j+1}$. Let $v=(v_1,\dots,v_n)$ be a point of $\partial U$. If $v$ has a zero coordinate,  let $i$ be the largest index such that  $v_i=0$. We have
	\begin{align*}
	f_i(x_1,\ldots,x_n)\,&=\,\frac{a}{1-x_i}\,-\,\frac{b}{x_i}\,-\,\sum_{j< i}\,\frac{c}{x_i-x_j}\,-\,\sum_{j> i}\,\frac{c}{x_i-x_j}\\
	&\leq \frac{a}{1-x_i}\,-\,\frac{b}{x_i}\,-\,\sum_{j> i}\,\frac{c}{x_i-x_j}.
	\end{align*}
	As
	\[
	\lim_{x\rightarrow v} \,\left(\frac{a}{1-x_i}-\sum_{j> i}\,\frac{c}{x_i-x_j}\right)\,=\,a+\sum_{j>i}\,\frac{c}{v_j}
	\]
which is finite, 	we see that $\disp \lim_{n\rightarrow\infty} f_i(u_n)=-\infty$.
	
	If one of the coordinates of $v$ is 1, let $i$ be the smallest index such that $v_i=1$, then 
	\begin{align*}
		f_i(x_1,\ldots,x_n)\,&=\,\frac{a}{1-x_i}\,-\,\frac{b}{x_i}\,-\,\sum_{j< i}\,\frac{c}{x_i-x_j}\,-\,\sum_{j> i}\,\frac{c}{x_i-x_j}\\
		&\geq \frac{a}{1-x_i}\,-\,\frac{b}{x_i}\,-\,\sum_{j< i}\,\frac{c}{x_i-x_j}.
	\end{align*}
		As
	\[
	\lim_{x\rightarrow v} \,\left(-\,\frac{b}{x_i}-\sum_{j< i}\,\frac{c}{x_i-x_j}\right)\,=\,-b-\sum_{j<i}\,\frac{c}{1-v_j}
	\]
	which is also finite,
	we have $\disp \lim_{n\rightarrow\infty} f_i(u_n)=\infty$.
	
If no coordinate of $v$ is 0 or 1, let $i$ be the smallest index such that $v_i=v_{i+1}$, then the first 3 terms of
	\[
	f_i(x_1,\ldots,x_n)\,=\,\frac{a}{1-x_i}\,-\,\frac{b}{x_i}\,-\,\sum_{j< i}\,\frac{c}{x_i-x_j}\,-\,\sum_{j> i}\,\frac{c}{x_i-x_j}\\
\]
have a finite limit while the last sum goes to $-\infty$. Thus $\disp \lim_{n\rightarrow\infty} f_i(u_n)=\infty$.
\end{proof}

\begin{thm}\label{syst1}
	The system \eqref{syst} has a solution in the cube $(0,1)^n$.
\end{thm}
\begin{proof}  The Jacobian of $f=(f_1,\dots, f_n)^t$ is given by
	\[
	D=\left(\frac{\partial f_i}{\partial x_j}\right)
	\]
	where, for $j\neq i$,
	\[
	\frac{\partial f_i}{\partial x_j}\,=\,\frac{c}{(x_i-x_j)^2}>0
	\]
	and
	\begin{align*}
	\frac{\partial f_i}{\partial x_i}\,&=\,\frac{a}{(1-x_i)^2}+\frac{b}{x_i^2}+\sum_{j\neq i}\,\frac{c}{(x_i-x_j)^2}\\
	&=\,\frac{a}{(1-x_i)^2}+\frac{b}{x_i^2}+\sum_{j\neq i}\,\frac{\partial f_i}{\partial x_j}.
	\end{align*}
	As the constants $a$, $b$ and $c$ are positive, we deduce that for $u\in\BR^n\setminus H$, the matrix $D(u)$ is diagonally dominant and therefore invertible. Now define 
	\[
	F=\sum_{i=1}^{n}\,f_i^2
	\]
	so that
	\[
	\begin{bmatrix}
	\frac{\partial F}{\partial x_1}\\.\\.\\.\\\frac{\partial F}{\partial x_n}
	\end{bmatrix}=2D f.
	\]
Let $\disp G(x)=(F(x)+1)^{-1}$. As $F(x)\geq 0$ and thanks to \lemref{lem61}, $G$ can be extended to a continuous map on the closure $\overline{U}$ which is compact, and having the value 0 at all the boundary points. Therefore $G$ has a global maximum at a point $u\in U$ which is also a local maximum. It follows that $u$ is a local minimum of $F$
 and thus $\disp \frac{\partial F}{\partial x_k}(u)=0$ for all $k$. In other words,  the gradient of $F$ vanishes, that is,   $2D(u)f(u)=0$. Since $D(u)$ is invertible, we have $f(u)=0$, so that $u$ is the desired solution. 		
	\end{proof}

\section{The general case.}
We have seen that $\eta^4$ is the only non-vanishing holomorphic weight 2 modular form with a character of $\SL$ (up to a scalar). The Schwarz derivative of its integral is a holomorphic weight 4 modular form for $\SL$. This can still happen if the weight 2 modular form is  non-vanishing and meromorphic as long as it has only double poles with zero residues. In this case, the integral is meromorphic with only simple poles to guarantee the holomorphicity of the Schwarz derivative. We start with $\eta^4$ to which we assign double poles using the elliptic modular $J-$function
\[
J(\tau)\,=\,\frac{E_4(\tau)}{1728\Delta(\tau)}\,,\ \ \Delta=\eta^{24}.
\]
Fix a positive integer $n$ and let $w_i$, $1\leq i\leq n$, be distinct points in $\fH$ that are not in the $\SL-$ orbits of $i$ or $\exp(2\pi i/3)$. Set
\[
f(\tau)\,=\,\frac{\eta^4(\tau)}{\prod_{i=1}^n\,(J(\tau)-J(w_i))^2}.
\]
The weight 2 form $f$ is non-vanishing on $\fH$ and for each $i$, $1\leq i\leq n$, $w_i$ is not a critical point of $J$ and hence  it is  a double pole of $f$. 	

For each $i$, $1\leq i\leq n$, set 
\[
\disp h_i(\tau)=f(\tau)(J(\tau)-J(w_i))^2 =\frac{\eta^4(\tau)}{\prod_{j\neq i}\,(J(\tau)-J(w_j))^2}.
\]
\begin{prop}\label{residue}
	For each $i$, $1\leq i\leq n$, the residue $\mbox{Res}(f,w_i)$ of $f$ at $w_i$ is given by
	\begin{equation}\label{resid}
\frac{h_i(w_i)}{6J'(w_i)}\left( \frac{3}{1-J(w_i)}-\frac{4}{J(w_i)}-\sum_{j\neq i}\frac{12}{J(w_i)-J(w_j)}\right).
\end{equation}
\end{prop}
\begin{proof}
	 Using the  Taylor expansion  of  $\disp f=h_i/(J-J(w_i))^2$ around $w_i$, one can easily show that the residue of $f$ at $w_i$ is given by
	\begin{align*}
\mbox{Res}(f,w_i)&=\frac{h_i'(w_i)}{J'(w_i)^2}-\frac{h_i(w_i)J''(w_i)}{J'(w_i)^3}\\
&=\frac{h_i(w_i)}{J'(w_i)^2}\left(\frac{4\eta'(w_i)}{\eta(w_i)}-\frac{J''(w_i)}{J'(w_i)}-\sum_{j\neq i}\,\frac{2 J'(w_i)}{J(w_i)-J(w_j)}\right).
	\end{align*}
On the other hand, we have the classical formula \cite[Chapter 6]{rankin}
\[
\Delta=\frac{-1}{(48\pi^2)^3}\frac{(J')^6}{J^4(J-1)^3}.
\]
Taking the logarithmic derivative, we get
\[
\frac{24\eta'}{\eta}=\frac{\Delta'}{\Delta}=\frac{6J''}{J'}-\frac{4J'}{J}-\frac{3J'}{J-1}.
\]
Therefore
\[
\frac{4\eta'}{\eta}-\frac{J''}{J'}=\frac{J'}{6}\left(\frac{3J'}{1-J}-\frac{4J'}{J}\right)
\]
and the proposition follows.
	\end{proof}
\begin{thm}\label{thm7.2}
For each positive integer $n$, there exist $w_1,\ldots w_n$ in $\fH$ such that
\begin{equation}\label{f-fn}
f_n(\tau)\,=\,\frac{\eta^4(\tau)}{\prod_{i=1}^n\,(J(\tau)-J(w_i))^2}
\end{equation}
is a non-vanishing weight $2$ modular form with a character and having a double pole  and  zero residue at each $w_i$  and holomorphic elsewhere in $\fH$ and at $\infty$.
\end{thm}
\begin{proof}
	This is a consequence of \thmref{syst1} with $a=3$, $b=4$ and $c=12$. Let $(x_1,\ldots x_n)$ be a solution to the system \eqref{syst}. For each $i$, there exists $w_i\in\fH$ such that
	$J(w_i)=x_i$. In fact, as $0<x_i<1$, $w_i$ can be taken in the arc of the unit circle between $i$ and $\exp(2\pi i/3)$. Using \propref{residue}, we see that $f_n$ satisfies the conditions of the theorem. It is easy to see that $f_n$ has a holomorphic $q-$expansion at $\infty$ with $q=\exp(2\pi i\tau/6)$ with leading term $\disp q^{1+12n}$.
	\end{proof}
As a consequence, the function
\begin{equation}
h_n(\tau)\,=\,\int_i^{\tau}\,f_n(z)\,dz
\end{equation}
is  $\SL-$equivariant with a triangular representation with at most simple poles and non-vanishing derivative. Therefore, its Schwarz derivative is a weight 4 modular form for $\SL$  that is holomorphic on $\fH$. Moreover, as the leading term of its $q-$expansion is $\disp \alpha q^{1+12n}$ with some constant $\alpha$, it is easy  to see that $\{h_n,\tau\}$ is holomorphic at $\infty$  with leading term $\disp 2\pi^2\frac{(12n+1)^2}{36}$. Including the holomorphic case with $n=0$, we conclude the following
\begin{thm}
For each integer  $n\geq 0$,	the function $h_n$  is a solution to
  \[
 \{h,\tau\}\,=\,sE_4(\tau) \quad \mbox{with }\ \  s=2\pi^2(12n+1)^2/36.
 \]
\end{thm}\qed

Let $w_i$, $1\leq i\leq n$, be as in \thmref{thm7.2}. We have
\begin{cor}
	For each integer $n\geq 0$, the function $\disp \eta^{-2}\prod_{i=1}^n(J(\tau)-J(w_i))$ is a solution to the differential equation
	\[
	y''\,+\,\frac{\pi^2(12n+1)^2}{36}E_4\,y\,=\,0.
	\]
\end{cor}\qed

Thus, we have been able to generalize the case $n=0$ due to Hurwitz and Klein to an arbitrary positive integer $n$.
\begin{remark}
	{\em The functions $f_n$ and $h_n$ are invariant under any action of the permutation group $S_n$ on the $w_i$ as well as if we change $w_i$ to $\gamma w_i$, $\gamma\in\SL$. However, we claim that $h_n$ and thus $f_n$ are unique up to a factor, and therefore the values $J(w_i)$ are unique up to a permutation. Heuristically, suppose that we have two solutions
		$\disp h_1=\eta^4/\prod_{i=1}^n(J-a_i)^2$ and 	$\disp h_2=\eta^4/\prod_{i=1}^n(J-b_i)^2$. Since $h_1$ and $h_2$ have the same Schwarz derivative, $h_1$ is a linear fraction of $h_2$. Meanwhile, both of their  $q-$expansions start with $\disp q^{1+12n}$, we deduce that $h_1=h_2$ and consequently one can deduce that the $a_i$'s are simply a permutation of the $b_i$'s. This suggests that the solution to \eqref{syst} might be unique, which is true at least for the case $a=3$, $b=4$ and $c=12$.
}
	\end{remark} 

{\bf Examples:}
	As we have seen, the case $n=0$ corresponds to the solution $f_n=\eta^4$. 
	
	For $n=1$, the system \eqref{syst} for $(a,b,c)=(3,4,12)$ is reduced to
	\[
	\frac{3}{1-x}-\frac{4}{x}=0
	\]
	which gives $x=7/4$ and thus
	\[
	f_1=\frac{\eta^4}{(J-7/4)^2}.
	\]
	For $n=2,\,3$ and $4$  and by a process of elimination, one can easily show that the solutions to the corresponding system \eqref{syst} are respectively given by the roots of the following polynomials:
	\begin{align*}
	n=2:\quad &  247x^2-260x+4 \\
	n=3:\quad & 31x^3-48x^2+\frac{96}{5}x-\frac{128}{95}\\
	n=4: \quad&   1233x^4-25234x^3+16368x^2-3520x+\frac{704}{5}.
\end{align*}
The Galois groups of these polynomials are given by the permutation group $S_n$ for the corresponding $n$. We conjecture that for each $n$, the solutions to \eqref{syst} are the roots of a degree $n$ irreducible  polynomial over $\BQ$ whose Galois group   is $S_n$.

{\bf Acknowledgment.} We thank David Handelman and Ahmed Sebbar for helpful discussions. 


\end{document}